\documentclass[12pt]{amsart}
\usepackage{amsmath,amssymb,amsfonts,amscd}

\usepackage{tikz}
\usetikzlibrary{arrows}
\usepackage{tikz-cd}
\usetikzlibrary{cd}

\newcommand{\lf}{\lambda_3}

\usepackage{array,multirow}
\usepackage{color}

\usepackage{hyperref}
\hypersetup{
    colorlinks = true,
    linkcolor = {black},
   citecolor  = {black},
   urlcolor = {black}
}

\usepackage{lmodern}
\usepackage{amsthm}
\usepackage{amssymb}
\usepackage{tikz}
\usetikzlibrary{decorations.pathreplacing}
\usetikzlibrary{patterns}

\newcommand{\ra}{\rightarrow}

\newcommand{\Q}{\mathbb{Q}}

\newcommand{\Z}{\mathbb{Z}}

\newtheorem{theorem}[subsection]{Theorem}
\newtheorem{lemma}{Lemma}[section]

\newtheorem{proposition}[lemma]{Proposition}

\theoremstyle{definition}
\newtheorem{definition}[lemma]{Definition}

\theoremstyle{remark}
\newtheorem{remark}[lemma]{Remark}
\numberwithin{equation}{section}
\theoremstyle{plain}

\theoremstyle{definition}

\usepackage{enumitem}

\theoremstyle{definition}

\newtheorem*{defin*}{Definition}

\theoremstyle{plain}

\begin{document}

\title[A triple torsion linking form and 3-manifolds in ${S^4}$]{A triple torsion linking form and 3-manifolds in $\mathbf{S^4}$}

\author{Michael Freedman}
\address{\hskip-\parindent
  Michael Freedman, 
  LogiQAI Solutions and 
CMSA, Harvard University, Cambridge, MA 02138
}
\email{mikehartleyfreedman@outlook.com}

\author{Vyacheslav Krushkal}
\address{\hskip-\parindent
  Slava Krushkal,
    Department of Mathematics\\
    University of Virginia\\
    Charlottesville, VA 22904}
\email{krushkal@virginia.edu}

\begin{abstract} Given a rational homology 3-sphere $M$, we introduce a triple linking form on $H_1(M; \Z)$, defined when the classical torsion linking pairing of three homology classes vanishes pairwise. If $M$ is the boundary of a simply-connected $4$-manifold $N$, the triple linking form can be computed in terms of the higher order intersection form on $N$, introduced by Matsumoto. We use these methods to formulate an embedding obstruction for rational homology spheres in $S^4$, extending a 1938 theorem of Hantzsche.
\end{abstract}

\maketitle

\section{Introduction}
The non-singular torsion linking pairing \[ \lambda\colon \tau H_k(M;\Z) \otimes \tau H_{n-k-1}(M;\Z) \to \Q/\Z, \]
where $M$ is a closed orientable $n$-manifold and $\tau$ denotes the torsion subgroup, is part of classical homology theory developed in the early 20th century \cite{Seifert, ST}.

Higher order versions of some of the key homological invariants were developed in the 1950s -- 1970s.
The Massey products \cite{UM, Massey} are higher order cohomological operations, defined when the cup product vanishes. More specific to low-dimensional topology, Milnor's invariants \cite{Milnor} are nilpotent invariants of links in 3-space generalizing the linking number, and the Matsumoto triple product \cite{Matsumoto} is defined on  homotopy classes of 2-spheres in 4-manifolds when pairwise intersection numbers are trivial. In fact, all three of these higher order invariants are closely related for a particular class of examples: 4-manifolds $N_L$ obtained from the 4-ball by attaching 0-framed 2-handles along a link $L$ with trivial linking numbers \cite{Turaev, Matsumoto}. 

We introduce a triple torsion linking form $\lf$ for rational homology 3-spheres. It is a $\Q/\Z$-valued invariant defined on three classes in $H_1(M^3;\Z)$ when the classical torsion linking pairing vanishes pairwise.\footnote{More precisely, $\lf(a,b,c)$ is defined for $a,b,c\in H_1(M^3;\Z)$ such that $\lambda(a,b)=\lambda(a,c)=0$.}
Our definition of $\lf$ is geometric, formulated in terms of intersections of certain 2-complexes, {\em rational gropes}, which can be found as a consequence of the vanishing of the classical torsion linking pairing, see Section~\ref{sec: triple linking}.

We show that $\lambda_3$ is an obstruction to embedding rational homology 3-spheres in $S^4$.
Recall the following theorem proved by Hantzsche in 1938 \cite{Hantzsche}, see  \cite{Freedman} for more details.

\begin{theorem}[Hantzsche]\label{theorem hantzsche}
	If an orientable closed 3-manifold $M^3$ embeds in $S^4$ there is a splitting $\tau H_1(M;\Z) \cong A \oplus B$ with $\lambda$ vanishing identically on both $A$ and $B$. Since $\lambda$ is non-singular, this gives natural isomorphisms $A \cong \operatorname{Hom}(B, \Q\slash\Z)$ and $B \cong \operatorname{Hom}(A,\Q\slash\Z)$. In particular, $|{\tau H_1(M;\Z)}|$ is a square integer. 
\end{theorem}

In this theorem, $A$ and $B$ are the kernels of the maps on $H_1$ induced by the inclusions of $M$ into the components of the complement of $M$ in $S^4$. The theorem may be restated as saying that these two kernels form a pair of {\em dual Lagrangians} with respect to the torsion linking pairing. 

We show that the triple torsion linking form gives a higher order embedding obstruction:

\begin{theorem} \label{thm: vanishing} 
 Suppose $M$ is a rational homology 3-sphere embedded in $S^4$. Then $\lf$ vanishes on each of the two dual Lagrangians $A,B$ in Hantzsche's theorem. 
\end{theorem}

The proof of Theorem \ref{thm: vanishing} is contained in Section \ref{sec: vanishing}. We give an example of a 3-manifold whose embedding in $S^4$ is obstructed by the new invariant $\lambda_3$ but not by Hantzsche's theorem.  This construction relies on the analysis of an interesting computational problem involving dual pairs of Lagrangians, see Section \ref{sec: examples}.

{\em Remark.} Both theorems \ref{theorem hantzsche}, \ref{thm: vanishing} hold for an integer homology 4-sphere in place of~$S^4$.

We relate $\lf$ to the Matsumoto triple product on a bounding 4-manifold in Section \ref{sec: relation Matsumoto}, extending the relation between torsion linking pairing and the intersection form recalled in Section \ref{sec: linking intersercion}. As a consequence, it is shown in Proposition \ref{prop: symmetry} that the triple torsion linking form is antisymmetric with respect to the action of the symmetric group $S_3$.

It is interesting to note that classical higher order invariants, discussed earlier in this introduction, are obstructions to embedding 4-manifolds $N_L=D^4\cup_{L, \, {\rm zero}\; {\rm framing}}$2-handles into $S^4$. This is easily seen from the fact that Milnor's invariants are concordance invariants of links, or alternatively using properties of the Matsumoto triple product. In general the boundary of a 4-manifold $N$ may embed in $S^4$ even when $N$ does not embed. In this sense Theorem \ref{thm: vanishing}  shows that our higher order invariant $\lf$ is a ``stronger'' obstruction to the 3-manifold embedding problem.

The work in this paper was motivated by the results of the first named author \cite{Freedman} which formulated the embedding problem for closed 3-manifolds in $S^4$ in terms of a criterion on Heegaard splittings. Our results do not directly contribute to the main motivation of \cite{Freedman}: distinguishing $S^4$ from a potential homotopy 4-sphere, since the triple linking pairing is an obstruction to embedding in a homology 4-sphere. Nevertheless, there is an interesting open question relating $\lambda_3$ to Heegaard splittings, see Section \ref{lf Heegaard}.

{\bf Convention.} All homology and cohomology groups will be taken with $\Z$ coefficients. The results of this paper apply in both smooth and topological categories.

{\em Acknowledgements.} 
We would like to thank Peter Teichner for his detailed feedback on this work, and in particular for his correction to the statement of Theorem \ref{thm: vanishing} and to the verification that the example in Section \ref{sec: examples} realizes our obstruction.
We also would like to thank Ryan Budney for his comments.

VK was supported in part by NSF grant DMS-2405044.

\section{The torsion linking pairing and the intersection pairing} \label{sec: linking intersercion}

For closed orientable 3-manifolds, it follows from Poincar\'{e} duality that
$H_1(M;\Z) \cong H^2(M;\Z)$,  
and $H^2$ fits in a universal coefficient sequence:
\begin{equation}\label{eq:iso-2}
	0 \ra \rm{Ext}_\Z(H_1(M;\Z),\Z) \ra H^2(M;\Z) \ra \rm{Hom}(H_2(M;\Z),\Z) \ra 0,
\end{equation}
where $\rm{Ext}_\Z(H_1(M;\Z),\Z)$ is naturally isomorphic to $\rm{Hom}_\Z(\tau H_1(M;\Z),  \, \Q/\Z).$
Combining these facts, there is a natural isomorphism $\tau H_1(M;\Z) \cong \operatorname{Hom}(\tau H_1(M),\Q \slash \Z)$, which may be reformulated as the nonsingular linking form:
\begin{equation}
	\tau H_1(M;\Z) \otimes \tau H_1(M;\Z) \xrightarrow{\lambda} \Q \slash \Z
\end{equation}

Geometrically, $\lambda$ may be computed as follows. Fix $t\in\Z\setminus 0$ such that $ta=0$ for all $a\in H_1(M)$. Given $[x],[y]\in H_1(M;\Z)$, let $\Sigma$ be a map of an oriented surface into $M$ with $\partial M=tx$. Then
\[ \lambda([x], [y])\, =\, \frac{1}{t}\, \Sigma\cdot y.\]

The facts that $\lambda$ is symmetric and independent of $t$ follow from the algebraic definition above.

The following method applies to any closed 3-manifold, so we state it in full generality, although we will use it only in the setting of rational homology spheres. Let $N$ be a 4-manifold with $H_1(N)=0$ and $\partial N=M$. Consider two elements $[x],[y]\in\tau H_1(M)$. Since $H_1 (N)=0$, there exist oriented surfaces $S_x, S_y$ in $N$ with $\partial S_x=x, \partial S_y=y$. As above, fix $t\in\Z\setminus 0$ such that $ta=0$ for all $a\in\tau H_1(M)$ and consider $\Sigma_x, \Sigma_y\subset M$ with $tx=\partial \Sigma_x, \, ty = \partial \Sigma_y$. 
Consider closed, oriented surfaces $X, Y$ in $N$:
\[ X:= \Sigma_x-tS_x, \, Y:= \Sigma_y-tS_y.
\]
The following relation between the torsion linking pairing in $M$ and the intersection pairing in $N$ can be found, for example, in \cite[Section 3]{GL}. 

\begin{lemma} \label{lem: linking intersection}  
\begin{equation} \label{eq: linking intersection}
\lambda([x],[y])\, =\, -\frac{1}{t^2} \, [X]\cdot [Y].
\end{equation}
\end{lemma}
The proof of Lemma \ref{lem: linking intersection} is illustrated in Figure \ref{fig:Linking_Intersections}.
\begin{figure}[ht]
\includegraphics[width=6cm]{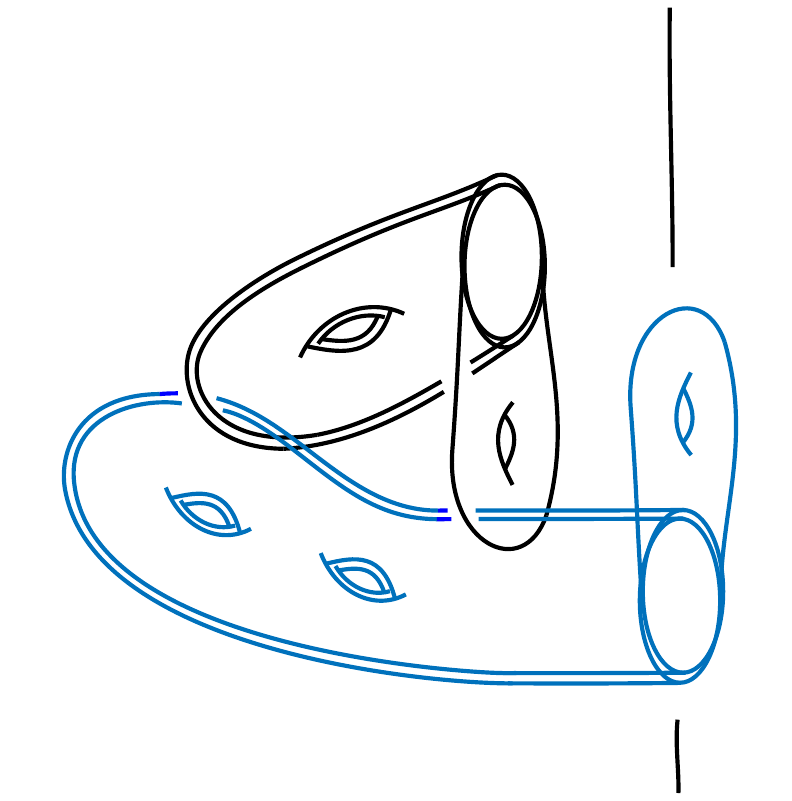}
\small
\put(-135,115){$tS_x$}
\put(-120,155){$N^4$}
\put(-22,155){$M^3=\partial N$}
\put(-67,139){$tx^{\uparrow}$}
\put(-88,70){$\Sigma_x^{\uparrow}$}
\put(-9,90){$\Sigma_y$}
\put(-150,32){$tS_y$}
\put(-16,20){$ty$}
\caption{Proof of Lemma \ref{lem: linking intersection}.}
\label{fig:Linking_Intersections}
\end{figure}

In more detail, push $x, \Sigma_x$ slightly into the interior of $N$, denoting the result by $x^{\uparrow}, \Sigma_x^{\uparrow}$. Denote $\Sigma_x^{\uparrow}-tS_x$ by $X^{\uparrow}$. Since $\Sigma_y$ is in $M$, $\Sigma_y\cap X^{\uparrow}=\emptyset$. Thus the intersection number $X^{\uparrow}\cdot Y$ is computed using intersections $tS_x\cap tS_y$ and $\Sigma_x^{\uparrow}\cap tS_y$.

Each intersection point $S_x\cap\kern-0.7em|\kern0.7em S_y$ gives rise to $t^2$ intersection points of the same sign between $tS_x$ and $tS_y$, so the contribution to the right hand side of \eqref{eq: linking intersection} is $0\in\Q/\Z.$ Finally, the intersections $\Sigma_x^{\uparrow} \cap S_y$ take place in a collar neighborhood of $\partial N$ where $S_y$ may be assumed to be a product $y\times I$, so $\Sigma_x^{\uparrow} \cdot tS_y =t(\Sigma_x\cdot y)$. \qed

\section{The triple torsion linking form} \label{sec: triple linking}
Throughout the rest of the paper $M$ will denote a rational homology 3-sphere. Surfaces used at various points in the construction will be mapped into $M$, in general they will not be embedded.
Fix $t\in \Z$ such that $tx=0$ for all $x\in H_1(M)$. Suppose $[x],[y],[z]$ are elements of $ H_1 (M)$ with 
\begin{equation} \label{eq: trivial linking}
    \lambda([x],[y])=\lambda([x],[z])=0.
\end{equation}  
Let $x, y, z$ be oriented curves embedded in $M$, representing their respective homology classes.
Consider a compact, oriented surface $\Sigma$ mapped into $M$ with $tx=\partial \Sigma$. Concretely, we view $tx$ as the curve $x$ with multiplicity $t$, so $\partial \Sigma$ wraps around it $t$ times, and a small oriented meridional circle $m$ to the curve $x$ has $t$ transverse intersections with the surface, all with the same sign. Because of the trivial linking assumption \eqref{eq: trivial linking}, the intersection numbers $\Sigma\cdot y$, $\Sigma \cdot z$ are elements of $t\, \Z$. 
Using finger moves on $y, z$ across $x$ (i.e. band-summing $y,z$ with $\pm m$), we find curves $y', z'$ homologous (in fact, isotopic) in $M$ to $y, z$ such that 
\begin{equation} \label{eq: intersection zero}
      \Sigma\cdot y'\, = \, \Sigma\cdot z' \, =\, 0.
\end{equation}
For brevity of notation, we will omit the primes and assume that $y, z, \Sigma$ satisfy \eqref{eq: intersection zero}. Surgering $\Sigma$ if necessary, we will assume $\Sigma\subset M\setminus (y\cup z)$.

\begin{definition} \label{def: triple linking}
    Consider homomorphisms $\phi, \psi\colon H_1(\Sigma)\to \Z$, defined as follows. Given $[\alpha]\in H_1(\Sigma)$, consider a $2$-chain $A$ in $M$ with $t\alpha=\partial A$. Define 
    \begin{equation} \label{phi} \phi([\alpha])\, =\, A\cdot y, \; \psi([\alpha])\, =\, A\cdot z. \end{equation} 
    The definition of $\phi, \psi$ is independent of a choice of $A$. Indeed, if $A'$ is another $2$-chain with $\partial A'=t\alpha$ then $A-A'$ is a 2-cycle which is null-homologous since $H_2(M)=0$, so the algebraic intersection number of $A-A'$ with any 1-cycle is trivial.
    
    The homomorphisms \eqref{phi} define cohomology classes $\Phi, \Psi\in H^1(\Sigma)\cong H^1(\Sigma, \partial \Sigma)$, and their cup product $\Phi\cup \Psi$ is an element of $H^2(\Sigma, \partial \Sigma)\cong \Z$. The triple linking is defined as 
    \begin{equation} \label{eq: triple linking}
        \lambda_3(x, y, z)\, :=\, \frac{1}{t} \, \Phi\cup \Psi\, \in\, \Q. 
    \end{equation}
\end{definition}

Before proving that $\lambda_3$ is well-defined as an element of $\Q/\Z$  on homology classes in Theorem \ref{thm: well defined}, we give a reformulation of the expression \ref{eq: triple linking}.
Let $g$ be the genus of $\Sigma$, and consider a symplectic basis of curves $\{ \gamma_i, \delta_i\}, i=1,\ldots, g$ on $\Sigma$. More precisely, for each $i$ choose an ordering of the pair $\gamma_i, \delta_i$, and endow these curves with some orientations, subject to the condition that the ordering and the orientations of $\gamma_i, \delta_i$ induce the given orientation of $\Sigma$.

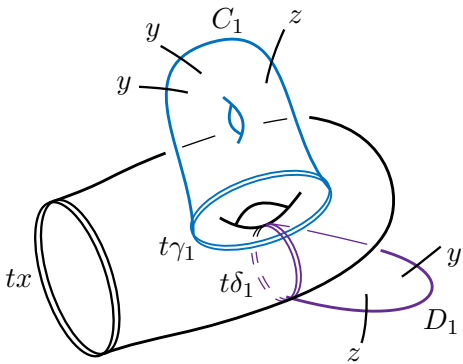
\begin{figure}[ht]
\includegraphics[width=6cm]{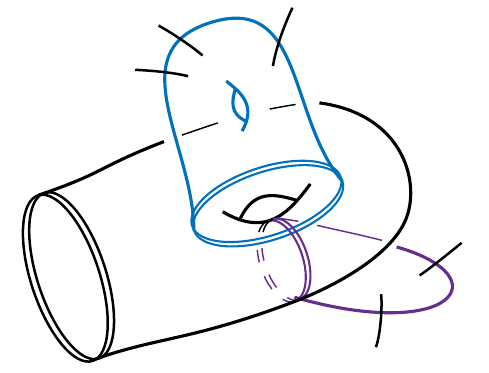}
\small
\put(-174, 30){$tx$}
\put(-116, 42){$t\gamma_1$}
\put(-93, 28){$t\delta_1$}
\put(-96, 128){$C_1$}
\put(-16, 14){$D_1$}
\put(-132, 105){$y$}
\put(-121, 125){$y$}
\put(-67, 130){$z$}
\put(-45,2){$z$}
\put(-7,40){$y$}

\caption{A schematic illustration of a genus one surface $\Sigma$, with $\partial C_1=t\gamma_1$, $\partial D_1=t\delta_1$. The 2-complex is mapped into $M$, and the intersection numbers of $C_1, D_1$ with $y,z$ are recorded in \eqref{eq: triple linking definition}.}
\label{fig: rational grope}
\end{figure}

Consider compact oriented surfaces $C_i, D_i$ mapped into $M$ with $\partial C_i= t\gamma_i, \partial D_i = t\delta_i$, Figure \ref{fig: rational grope}. 
An expression for \eqref{eq: triple linking}, with respect to the given symplectic basis of $\Sigma$, is given by
\begin{equation} \label{eq: triple linking definition}
    \lf(x, y, z)\, :=\, \frac{1}{t}\, \sum_{i=1}^g  (C_i\cdot y)(D_i\cdot z) - (C_i\cdot z)(D_i\cdot y) \in\Q. 
\end{equation}

Note that \eqref{eq: triple linking definition} is independent of the choice on an ordering and orientations of $\gamma_i, \delta_i$, as long as they induce the given orientation of $\Sigma$. In the following proof, it will be convenient to specify the surface $\Sigma$ used in the formula on the right hand side of \eqref{eq: triple linking definition}, and in this case we will use the notation $\lf(\Sigma, y, z)$.

\begin{remark} \label{rem: lf} \
\begin{enumerate} 
\item 
The 2-complex assembled of $\Sigma$ and the collection of surfaces $C_i, D_i$ may be thought of as a rational version of a grope \cite[Chapter 2]{FQ}. 
\item 
Unlike the classical torsion linking pairing $\lambda$, the triple linking form $\lf$ in general depends on the choice of $t$: $\lf$ scales as $t^3$, while the factor $\frac{1}{t}$ in \eqref{eq: triple linking}, \eqref{eq: triple linking definition} is forced by the well-definedness proof below.  
\end{enumerate}
\end{remark}

\begin{theorem}\label{thm: well defined}
    Given three classes $[x],[y],[z]\in H_1(M;\Z)$ with \[ \lambda([x],[y])=\lambda([x],[z])=0,\] 
    the triple torsion linking $ \lf([x],[y],[z])\in \Q/\Z$ is well-defined.
\end{theorem}

The proof of Theorem \ref{thm: well defined} follows from the following lemma.

\begin{lemma} \label{lem: well defined}
The triple linking form in equation \eqref{eq: triple linking definition} is independent of a choice of:
\begin{enumerate}[label=(\alph*)]
\item\label{1} surfaces $C_i, D_i$ with boundary $t\gamma_i, t\delta_i$,
\item\label{2} a symplectic basis $\{ \gamma_i, \delta_i\} $ of $S$,
\item\label{3} a bounding surface $\Sigma$, $\partial \Sigma=t\alpha$, 
\item\label{4} a curve $x$ representing its homology class, 
\item\label{5} curves $y, z$ within their homology classes.
\end{enumerate}
\end{lemma}

\begin{proof}

Items \ref{1} and \ref{2} follow from the fact that $H_2(M)=0$ and from the basis-independent formula \eqref{eq: triple linking}.

{\em Proof of \ref{3}}. Consider $y,z$ as oriented curves embedded in $M$.
Since $H_2(M)=0$, $H_2(M\setminus (y\cup  z))\cong\Z$ is generated by the homology class $[T]$ of the torus boundary $T$ of a tubular neighborhood of either one of the two curves, say $y$. A symplectic basis of curves in $T$ is given by a meridian $\mu$ and a longitude $\lambda$ of $y$. Consider the analogue of the formula \eqref{eq: triple linking definition} for the closed genus 1 surface $T$, where $\partial C=t\mu, \partial D=t\lambda$. The result is trivial in $\Q/\Z$ since $C$ may be taken to be $t$ copies of the meridional disk bounded by $\mu$, and hence $C\cdot y=\pm t$ and $C\cdot z=0$.

Let $\Sigma'$ be another oriented surface with $\partial \Sigma'=t x$. Then the homology class of the 2-cycle $\Sigma\cup -\Sigma'$ equals a multiple $k[T]$ in $H_2(M\setminus (y\cup z))$. A connected sum $\overline \Sigma:=(\Sigma\cup -\Sigma')\# (-kT)$ is null-homologous in $M$; by an application of the Atiyah-Hirzebruch spectral sequence it follows that $\overline \Sigma$ bounds a map of an oriented $3$-manifold $N$ into $M$. 

Let $\overline g$ be the genus of $\overline \Sigma$. By Poincar\'{e} duality, ${\rm ker}[H_1(\overline\Sigma; \Q)\to H_1(N; \Q)]$ is a rational Lagrangian subspace $L_{\Q}$ of $H_2(\overline\Sigma; \Q)$, in particular the rank of $L_{\Q}$ is $\frac{1}{2} \, {\rm rk}(H_2(\overline\Sigma; \Q))$. The intersection of $L_{\Q}$ with $H_2(\overline\Sigma; \Z) \subset H_2(\overline\Sigma; \Q)$ is an integral Lagrangian $L$: a sublattice of $H_2(\overline\Sigma; \Z)$ with $H_2(\overline\Sigma; \Z)/L\cong \Z^{\overline g}$. By \cite[Chapter 6]{FM} its basis may be represented by a {\em geometric} Lagrangian: a collection of pairwise disjoint, homologically independent simple closed curves $\sigma_1,\ldots, \sigma_{\overline g}$ in $\overline\Sigma$.

Consider a surface $S_i$ bounded by $t\sigma_i$ in $M$ for each $i$. 
Since $[\sigma_i]=0\in H_1(N, \Q)$, a multiple $m t\, \sigma_i$ bounds in $N$ integrally for some $m\in\Z\setminus 0$. On one hand, $m t\, \sigma_i$ bounds $m S_i$ in $M$, on the other hand, it also bounds a surface $S'_i$ in $N$.  
Since $N$ is mapped into $M\setminus (x\cup y)$, the $\Z$-valued algebraic intersection numbers $S'_i\cdot y=-m\, (S_i\cdot y), \, S'_i\cdot z=-m\, (S_i\cdot z)$ are trivial. Therefore $S_i\cdot y= S_i\cdot y=0$. 

Completing $\{ \sigma_i\}$ to a symplectic basis of $\overline\Sigma$, the expression $\lf(\overline\Sigma, y, z)$ in \eqref{eq: triple linking definition} for this symplectic basis is trivial.
It was observed in \ref{2} that $\lf(\overline\Sigma, y, z)$ is independent of a symplectic basis, so it can be computed as a sum of three terms corresponding to $\Sigma, -\Sigma', kT$ respectively. 
Moreover, it was shown above that the term corresponding to $kT$ is zero. Hence the expressions $\lf(\Sigma, y, z)$, $\lf(\Sigma', y, z)$ are equal.

{\em Proof of \ref{4}}. Suppose $x_1, x_2$ are homologous, with 
\[
tx_i=\partial \Sigma_i, \; \Sigma_i\cap y=\Sigma_i\cap z=\emptyset, \; i=1,2.
\]
Consider an oriented surface $S$ with $\partial S=x_2-x_1$. Then $\widetilde \Sigma:=\Sigma_1+tS-\Sigma_2$ is a closed surface; it is null-homologous since $H_2(M)=0$. It follows that
\[
0\, =\, \widetilde \Sigma\cdot y\, =\, t(S\cdot y)
\]
where $t\neq 0$, so $S\cdot y=0$, and analogously $S\cdot z=0$. Surgering it if necessary, $S$ may be assumed to be disjoint from $y, z$. Moreover, the multiple $t(x_1-x_2)$ bounds $tS$, and $\lf(tS, y, z)=0\in \Q/\Z$ because of the factor $t$ in $tS$. Now $tx_1=\partial \Sigma_1=\partial (\Sigma_2-tS)$. It follows from \ref{3} that 
\[ \lf(\Sigma_1, y, z)=\lf(\Sigma_2-tS, y, z)=\lf(\Sigma_2, y, z).
\]

{\em Proof of \ref{5}.} Let $tx=\partial \Sigma$, and suppose $y, y', z, z'$ are curves in $M\setminus \Sigma$ with $[y]=[y'], [z]=[z']\in H_1(M)$. Consider a factor, say $C_i\cdot y$, in one of the terms in \eqref{eq: triple linking definition}. Its multiple $\frac{1}{t} (C_i\, \cdot\,  y)$ equals the torsion linking pairing $\lambda([\gamma_i],[y])$. Since the value of this pairing in $\Q/\Z$ is well-defined on homology classes, replacing $y$ with a homologous $y'$ changes $C_i \cdot  y$ by an element of $t\Z$. The overall quantity \eqref{eq: triple linking definition} is unchanged, as an element of $\Q/\Z$.

This concludes the proofs of Lemma \ref{lem: well defined} and of Theorem \ref{thm: well defined}.
\end{proof}

\section{The Matsumoto triple product} \label{sec: Matsumoto} 
In this section we recall the definition of the Matsumoto product and give a reformulation in terms of intersections of capped surfaces. This will serve as a background for the proofs in Section \ref{sec: vanishing}.

We start by recalling the definition from \cite[Section 3]{Matsumoto}. Let $N$ be a compact oriented 4-manifold with $H_1(M)=0$, and let $X_1, X_2, X_3\in \pi_2(N)$ be homotopy classes with $[X_i]\cdot[X_j]=0$ for $i\neq j\in\{1,2,3\}$. For each $i$, let $f_i\colon S^2\to N$ be a map in the homotopy class $X_i$, so that the three spheres $S_i:=f_i(S^2)$ are in general position. 

Since the intersection numbers vanish, the intersection points $S_i\cap S_j$ can be paired up by disjoint Whitney arcs in both $S_i$ and $S_j$. The orientation of the Whitney circles is defined as follows. For each ordered pair $(i,j)\in \{ (1,2), (2,3), (3,1)\}$, the Whitney arc for the intersections $S_i\cap S_j$ is oriented from positive to negative on $S_i$ and from negative to positive on $S_j$. 

Since $H_1(N)=0$, the Whitney circles for $S_i\cap S_j$ bound oriented surfaces denoted $\{ \Delta^{n}_{i,j}\}$; set $ \Delta_{i,j} := \cup_n  \Delta^{n}_{i,j}$. The Matsumoto triple product is defined as\footnote{The definition in \cite{Matsumoto} does not assume that the Whitney arcs are disjoint and the formula there has additional terms involving intersections between the arcs. We will not use this more general definition.}
\begin{equation} \label{eq: Matsumoto} \langle X_1,X_2,X_3\rangle \, :=\, S_1\cdot \Delta_{2,3}+S_2\cdot \Delta_{3,1}+S_3\cdot \Delta_{1,2}.
\end{equation}
It is shown in \cite[Proposition 4]{Matsumoto} that $\langle X_1,X_2,X_3\rangle$ is well defined in $\Z/I$ where $I$ is the indeterminacy ideal in $\Z$ defined by the intersection numbers $[X_i]\cdot Y$ for $i=1,2,3$ and all $Y\in H_2(M)$. Here the indeterminacy is present due to the choice of surfaces $\Delta_{i,j}$: the difference of two choices $\Delta_{i,j}-\Delta'_{i,j}$ is a 2-cycle whose algebraic intersection number with $S_k$ may affect \eqref{eq: Matsumoto}.

\begin{figure}[ht]
\centering
\includegraphics[height=3.6cm]{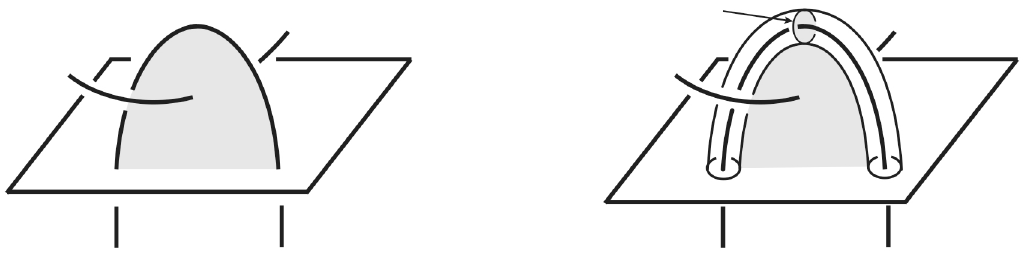}
{\small     
    \put(-294,6){${S}^{}_j $}
       \put(-280,63){${S}^{}_i $}
       \put(-400,78){${S}^{}_k $}
       \put(-330,44){$\Delta^n_{ij}$}
        \put(-50,6){${R}^{}_j $}
       \put(-36,63){${R}^{}_i $}
       \put(-155,78){${R}^{}_k $}
       \put(-107,42){$B^i_n$}
       \put(-96,104){$\alpha^i_n$}
       \put(-81,50){$\beta^i_n$}
       \put(-137,97){$A^i_n$}
    }  
    \caption{From 2-spheres $S_i$ to gropes of height 2.}
\label{fig: grope}
\end{figure}
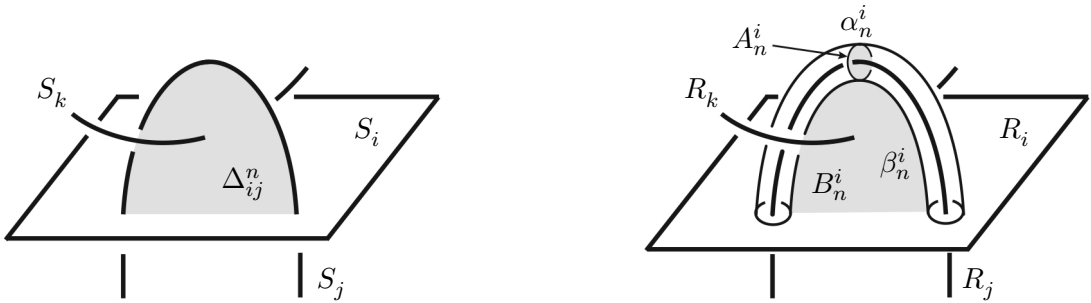

Next we derive an expression for the Matsumoto product in a format that makes it easier to relate to $\lambda_3$.

The intersections of spheres in the construction above may be resolved by increasing genus; denote the surface obtained from $S_i$ by $R_i$, $i=1,2,3$. Second stage surfaces $ A^i_n, B^i_n$ can be added to a symplectic basis of curves $\{ \alpha^i_n, \beta^i_n \}$ of $R_i$: a meridional disk $A^i_n$ and the surface $B^i_n$ equal to $\Delta^n_{ij}$ minus a collar on its boundary, attached to a dual curve corresponding to the Whitney circle, Figure \ref{fig: grope}. These are gropes of height 2 mapped into $N$, \cite[Section 2]{FQ}. There are two ways to resolve each pair of intersections $S_i\cap S_j$: by adding genus to $S_i$ or $S_j$, and the choice is recorded in the superscript of the second stage surfaces. Let the surface

Consider an ordering and orientations of each pair $ \alpha^i_n, \beta^i_n $ as described in the paragraph following \eqref{eq: triple linking} in Section \ref{sec: triple linking}.
Let $g$ be the genus of $R_i$. 
Define
\begin{equation} \label{eq:  Matsumoto x}
    \langle R_i, R_j, R_k\rangle^i\, :=\, \, \sum_{n=1}^g  (A^i_n\cdot R_j)(B^i_n\cdot R_k) - (A^i_n\cdot R_k)(B^i_n\cdot R_j) \in\Z. 
\end{equation}
and consider the analogous expressions with superscripts $j,k$. (The superscripts are included to distinguish the notation from the Matsumoto product.)
\begin{lemma} \label{lem: Matsumoto} The Matsumoto triple product can be computed using \eqref{eq:  Matsumoto x}:
    \begin{equation} \label{eq: Matsumoto equality} 
    - \langle X, Y, Z\rangle \, =\,
    \langle R_1, R_2, R_3\rangle^1\,+ \, \langle R_2, R_3, R_1\rangle^2\, + \, \langle R_3, R_1, R_2\rangle^3.
\end{equation}
\end{lemma}

\begin{proof}
As shown in Figure \ref{fig: grope}, there is a single intersection point of the meridional disk $A^i_n$ with $R_j$, $A^i_n\cap R_k=\emptyset$, and $R_k\cap B_n=S_k\cap\Delta_{i,j}$. 
The minus sign in \eqref{eq: Matsumoto equality} is present due to the fact that our orientation convention is opposite of that in \cite{Matsumoto}.
\end{proof}

\begin{remark} \label{rem: Matsumoto} \
(1) In Lemma \ref{lem: Matsumoto} the surfaces $R_i$ were constructed starting from the spheres $S_i$. The formula \eqref{eq: Matsumoto equality} can be used to compute
the Matsumoto product in a simply-connected 4-manifold $N$ more generally as follows. Consider three classes in $H_2(N)\cong\pi_2(N)$ with pairwise intersection numbers equal to zero. Represent them by disjoint maps of surfaces $R_1, R_2, R_3$, and using the simple connectivity assumption choose maps of disks for second stage surfaces $A, B$. The result is a collection of capped surfaces. Then the equality \eqref{eq: Matsumoto equality} holds because surgering the bodies of the capped surfaces along one set of caps to get 2-spheres, parallel copies of the dual caps serve as Whitney disks for the resulting double points.
\end{remark}    

(2) The expression \eqref{eq:  Matsumoto x} is closely related to the higher order intersection invariant  \cite[Remark 9]{ScT} taking values in the space of trees modulo antisymmetry and Jacobi relations.

\section{Proof of Theorem \ref{thm: vanishing}} \label{sec: vanishing}
Let $S^4\setminus \mathcal{N}(M)=A\sqcup B$. By Hantzsche's theorem, there are two Lagrangian subspaces in $H_1(M)$, kernels of the inclusion maps into $A$ and into $B$. To be concrete, let $[x], [y],[z]\in{\rm ker}[H_1(M)\to H_1(A)]$.

Note that the inclusion map induces a surjection $H_1(M) \twoheadrightarrow H_1(A)$, so any choice of $t\in \Z$ such that $tx=0$ for all $x\in H_1(M)$ works in $H_1(A)$ as well. 

Consider three classes $[X], [Y], [Z]\in H_2(A)$ with pairwise trivial intersection numbers, and consider oriented surface representatives $X, Y, Z$. We will now formulate a version $\langle [X], [Y], [Z]\rangle_{\Q}$ of the Matsumoto triple pairing, adapted to the setting of a rational homology ball.

Since the intersection numbers vanish, we may surger the surfaces if necessary to find homologous representatives (still denoted $X, Y, Z$) which are pairwise disjoint. As in Section \ref{sec: triple linking}, consider a symplectic basis of curves $\{ \gamma^x_i, \delta^x_i\}, i=1,\ldots, g_x$ on $X$, and compact oriented surfaces $C^x_i, D^x_i$ mapped into $A$ with $\partial C^x_i= t\gamma^x_i, \partial D^x_i = t\delta^x_i$. 
In analogy with gropes \cite[Section 2]{FQ}, we call $X$ the {\em body} and $C_i, D_i$ {\em second stage surfaces} of this construction. Consider a version of equation \eqref{eq: triple linking definition} without the $\frac{1}{t}$ factor:
\begin{equation} \label{eq: rational Matsumoto x}
    \langle X, Y, Z\rangle^x_{\Q}\, :=\, \, \sum_{i=1}^g  (C^x_i\cdot Y)(D^x_i\cdot Z) - (C^x_i\cdot Z)(D^x_i\cdot Y) \in\Z. 
\end{equation}
A basis-independent definition can be given as in \eqref{eq: triple linking}.

\begin{remark}
We emphasize the difference of equations \eqref{eq: rational Matsumoto x} and \eqref{eq:  Matsumoto x}. The expression \eqref{eq:  Matsumoto x} 
 for the Matsumoto product uses surfaces $A^i_n, B^i_n$ bounded by curves in a symplectic basis on a surface. In a rational homology ball - the setting for \eqref{eq: rational Matsumoto x} - in general these curves do not bound, and the formula uses surfaces $C^x_i, D^x_i$ bounded by $t$ multiples of the curves.
 \end{remark}

Considering the analogous expressions for $y,z$, define the ``rational Matsumoto pairing'' 
\begin{equation} \label{eq: rational Matsumoto}
    \langle X, Y, Z\rangle_{\Q}\, :=\, \, \langle X, Y, Z\rangle^x_{\Q}\,+ \, \langle Y, Z, X\rangle^y_{\Q}\, + \, \langle Z, X, Y\rangle^z_{\Q}. 
\end{equation}

\begin{remark} \label{rem: defined on cycles}
We do not know if $\langle X, Y, Z\rangle_{\Q}$ is well-defined on homology classes; \eqref{eq: rational Matsumoto} will be used below as a quantity associated with particular representatives $X, Y, Z$.
\end{remark}

Following \cite{GL} and Section \ref{sec: linking intersercion}, consider surfaces $\Sigma_x, S_x$ mapped into $M$ and $A$ respectively with $\partial \Sigma_x=tx, \partial S_x=x$. We also consider the analogous surfaces for $y,z$.

Consider the following three closed surfaces mapped into $A$: 
\begin{equation} \label{eq: 2cycles}
X:= \Sigma_x-tS_x, \, Y:= \Sigma_y-tS_y, \,  Z:= \Sigma_z-tS_z.
\end{equation}

Push $\Sigma_x$ slightly into the interior of $A$, and denote the result by $\Sigma_x^{\uparrow}$ (and the effect on $X$ by $X^{\uparrow}$). 
We have the following analogue of the relation between the torsion linking pairing and the intersection pairing, Lemma \ref{lem: linking intersection}:

\begin{lemma} \label{lem: triple linking Matsumoto} Let $X, Y, Z$ be associated with $x, y, z$ as above. Then the equality
\[\lf ([x],[y],[z])\, =\, \frac{1}{t^3} \, \langle X^{\uparrow}, Y, Z\rangle_{\Q}. \]
holds in $\Q/\Z$.
\end{lemma}

{\em Proof of Lemma \ref{lem: triple linking Matsumoto}.} The setting is similar to that in Figure \ref{fig:Linking_Intersections}, except that now there is also the third surface $Z$, and the intersections are all between bodies and second stage surfaces. 

We start by noting that all second stage surfaces of $Y, Z$ that are attached to curves in $\Sigma_y, \Sigma_z$ are in M. therefore they are disjoint from the body of $X$ and do not contribute to \eqref{eq: rational Matsumoto}. 

Next, intersections ``deep in $N$'', corresponding to $S_x, S_y, S_z$ and their second stage surfaces, contribute multiples of $\pm t^3$ to the expression of the form \eqref{eq: rational Matsumoto x}. More precisely, the multiple $t$ of the base surface of $X$ has the effect of replacing each pair $\gamma^x_i, \delta^x_i$ with $t$ parallel copies, and taking $t$ times $S_y, S_z$ contributes an additional factor of $t^2$. Thus multiplied by $\frac{1}{t^3}$, the contribution of these intersections is trivial in $\Q/\Z$.

The remaining intersections are between the second stage surfaces attached to $\Sigma_x^{\uparrow}$ and $tS_y, tS_z$. Using the product structure of $S_x, S_z$ in a collar on $\partial N$, the formula in \eqref{eq: rational Matsumoto x} is seen to be identical to that in \eqref{eq: triple linking definition}. The fact that there are $t$ copies of both $S_y$ and $S_z$, and the factor $\frac{1}{t}$ in the definition of $\lf$ gives the factor $\frac{1}{t^3}$ in Lemma \ref{lem: triple linking Matsumoto}.
\qed

\begin{lemma} \label{lem: rational Matsumoto zero} Let $X^{\uparrow}, Y, Z\subset A\subset S^4$ be associated with $x, y, z$ as above. Then
\[ \langle X^{\uparrow}, Y, Z\rangle_{\Q}\, =\, t^2 \, \langle [X], [Y], [Z]\rangle \, =\, 0. \]
\end{lemma}

Here $\langle [X], [Y], [Z]\rangle$ is the Matsumoto pairing in $S^4$, which is trivial.

{\em Proof of Lemma \ref{lem: rational Matsumoto zero}.} The curves $\gamma^x_i, \delta^x_i$  bound surfaces $C'_i, D'_i$ in $S^4$. Then $t\gamma^x_i=t\partial C'_i, t\delta^x_i=t\partial D_i$.  
Since $Y, Z$ are $2$-cycles and $C_i^x-t\partial C'_i$, $D_i^x-t\partial D'_i$ are null-homologous, each factor on the right hand side of 
\eqref{eq: rational Matsumoto x} is unchanged when $C_i^x, D_i^x$ are replaced with $t\, C'_i, t\, D'_i$. Computing the quantity \eqref{eq: rational Matsumoto x} with $C'_i, D'_i$ gives the Matsumoto product (see Remark \ref{rem: Matsumoto}), and the above discussion gives the factor $t^2$.
\qed

\section{Symmetry and relation with the Matsumoto triple product} \label{sec: relation Matsumoto} 
We state the relation of $\lf$ with the Matsumoto product, which is implicit in the proof of Theorem \ref{thm: vanishing}.

\begin{proposition} \label{prop: l3 Matsumoto}
Given $x, y, z\in H_1(M)$, suppose $N$ is a simply-connected 4-manifold with $\partial N=M$, such that the three associated homology classes $[X], [Y], [Z]\in H_2(N)$ have trivial pairwise intersection numbers.
Then
\begin{equation} \label{eq: l3 M}
\lf([x],[y],[z])\, =\, \frac{1}{t}\langle [X], [Y], [Z]\rangle.
\end{equation}
\end{proposition}

The proof is a consequence of the proofs of Lemmas \ref{lem: triple linking Matsumoto},  \ref{lem: rational Matsumoto zero}.

{\em Remark.} The Matsumoto product $\langle X, Y, Z\rangle$ is well-defined in $\Z/I$ where $I$ is the ideal generated by intersection numbers of $X, Y, Z$ with other  classes in $H_2(N)$. Since $X, Y, Z$ are assembled of surfaces $\Sigma\subset M$ and $tS$ in the interior of $N$, it follows that the intersection of any class in $H_2(N)$ with $X, Y, Z$ is an element of $t\Z$. By Proposition \ref{prop: l3 Matsumoto}, the indeterminacy does not affect the relation with $\lf$ in \eqref{eq: l3 M}.

\begin{proposition} \label{prop: symmetry}
    Let $a_1,a_2,a_3\in H_1(M)$ with $\lambda(a_i,a_j)=0$ for all $i\neq j$ and consider a permutation $\sigma\in S_3$. Then
    \[ \lf(a_1, a_2, a_3)\, =\, {\rm sgn(\sigma)}\, \lf(a_{\sigma(1)}, a_{\sigma(2)}, a_{\sigma(3)}).\]
\end{proposition}

\begin{proof}
First we claim that given $M$ and $a_1,a_2,a_3\in H_1(M)$ as in the statement, there is a simply-connected 4-manifold $N$, $\partial N=M$ so that the classes $A_i\in H_2(N)$ corresponding to $a_i$ satisfy $A_i\cdot A_j=0$ for all $i\neq j\in\{1, 2, 3\}$.

As in \eqref{eq: trivial linking}, using the assumption $\lambda(a_i,a_j)=0$ we find curves $x, y, z$ representing the homology classes $a_1, a_2, a_3$ such that $\Sigma_x\cdot y = \Sigma_x\cdot z =0$, with the analogous condition where $x,y,z$ are cyclically permuted.
Now attach $2$-handles to $M\times I$ with any framing to the curves $x,y,z$ in $M\times \{ 1\}$, and then fill in the resulting $3$-manifold with a simply-connected $4$-manifold.

Recall from the proof of Lemma \ref{lem: linking intersection}  that there are three contributions to the intersection number. Following the notation of that lemma, $\Sigma_y\cap X^{\uparrow}=0$. The intersections deep in $N$  are zero by construction of the 4-manifold $N$ because the curves $x,y,z$ bound disjoint disks (cores of the 2-handles). The last contribution equals $\Sigma_x\cdot y=0$. 

Now the result follows from Proposition \ref{prop: l3 Matsumoto} since the definition \eqref{eq: Matsumoto} of the Matsumoto product has the required symmetry with respect to permutations.
\end{proof}

\section{An example} \label{sec: examples}
The lens space $L(3,1)$ does not embed in an integer homology 4-sphere (this follows for example from Hantzsche's theorem) but a punctured $L(3,1)$ embeds in $S^4$ \cite{Zeeman}. It follows that $L(3,1)\# -L(3,1)$ embeds in $S^4$. A Kirby diagram for this connected sum is given by a two component unlink with the framing of one component equal to $3$ and the other one $-3$. 

\subsection{A preliminary construction} \label{sec: preliminary}
Consider the framed link in Figure \ref{fig:Example0} 
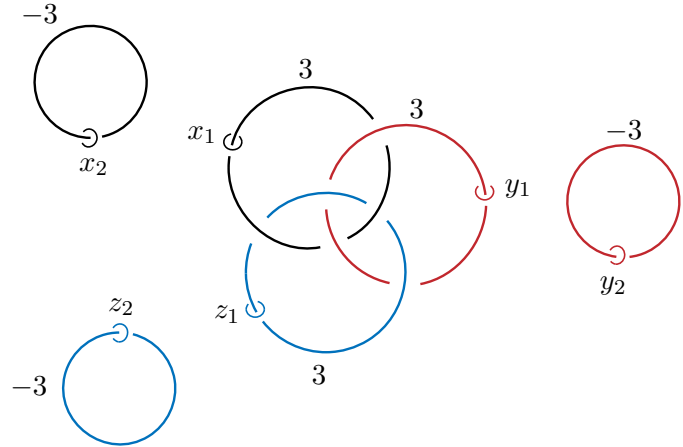
\begin{figure}[h]
\includegraphics[width=9cm]{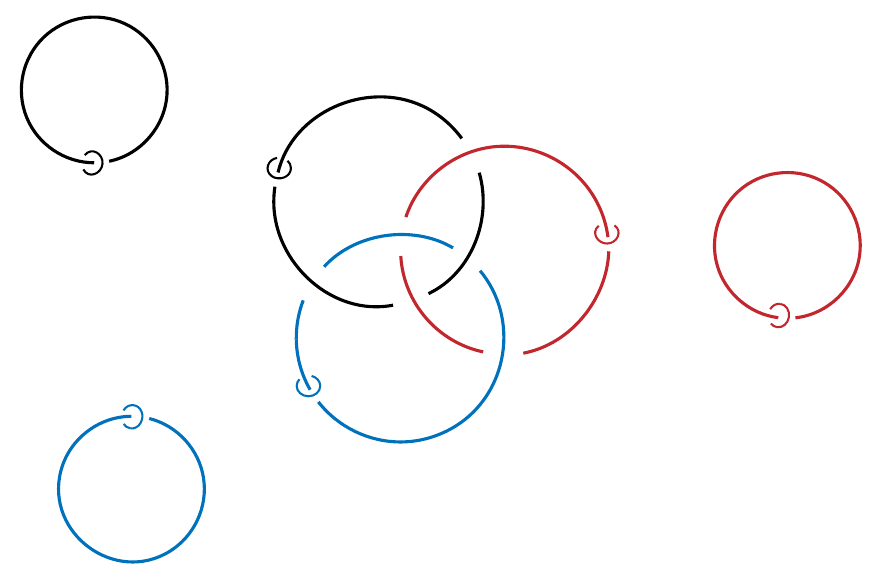}
\small
\put(-255,164){$-3$}
\put(-259,23){$-3$}
\put(-34,120){$-3$}
\put(-150,143){$3$}
\put(-108,128){$3$}
\put(-145,27){$3$}
\put(-233,108){$x_2$}
\put(-192,120){$x_1$}
\put(-36,63){$y_2$}
\put(-72,100){$y_1$}
\put(-222,55){$z_2$}
\put(-182,52){$z_1$}
\caption{A Kirby diagram for the manifold $M_{0}$.}
\label{fig:Example0}
\end{figure}
with three $3$-framed curves forming the Borromean rings and three $(-3)$-framed unlinked circles. 
This framed link defines a 3-manifold $M_{0}$ whose torsion linking pairing is isomorphic to that of a connected sum of three copies of $L(3,1)\# -L(3,1)$. It follows that Hantzsche's theorem does not give an obstruction to embeddability of $M_{0}$. Next we analyze the triple torsion linking pairing on the Lagrangians in the first homology of this manifold.

The first homology of $M_0$ is $(\Z/3\Z)^6$, generated by meridians to the link components. Every element is a torsion element of order 3, and we will fix $t=3$ throughout this section. (Compare with Remark \ref{rem: lf} (2).)

The left part of Figure \ref{fig:Bor} gives a more detailed illustration of the surgery on the Borromean rings. The figure shows three solid tori, a regular neighborhood of the Borromean rings, that are removed from the 3-sphere, and whose boundary tori are filled in with slope 3. Let $l_{x_1}, l_{y_1}, l_{z_1}$ denote longitudes of the tori; meridional circles linking the tori are denoted in Figures \ref{fig:Example0}, \ref{fig:Bor} by $x_1, y_1, z_1$. Consider the orientations of the longitudes of the link components in Figure \ref{fig:Bor} and the induced orientations of the meridians using the right hand rule.

\begin{figure}[ht]
\includegraphics[width=15cm]{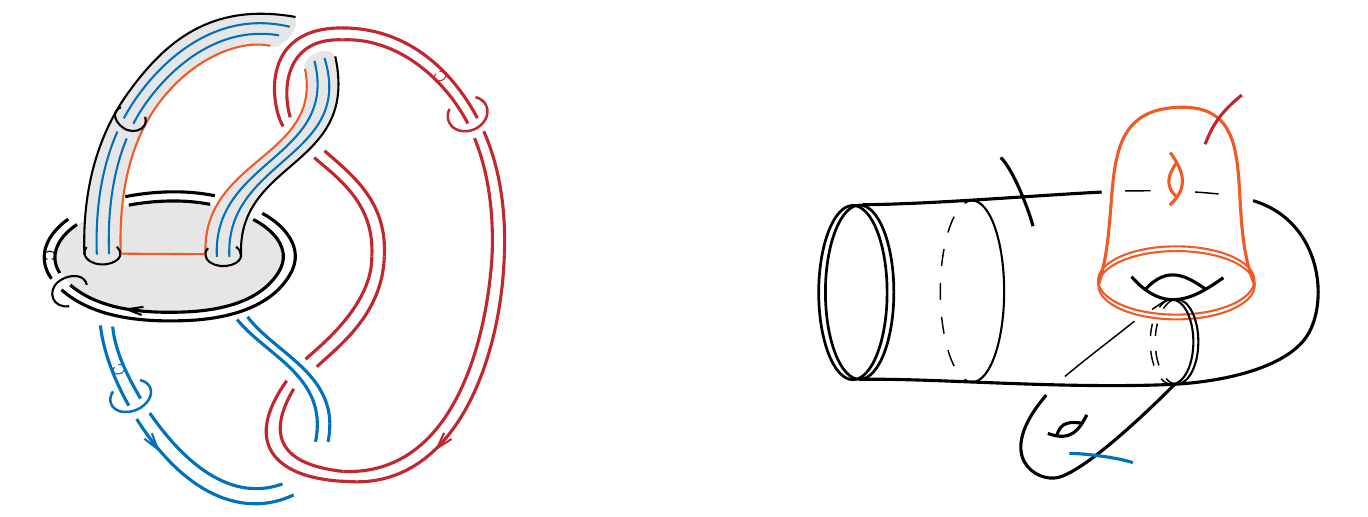}
\small
\put(-419, 59){$x_1$}
\put(-401,124){$\gamma$}
\put(-368, 130){$\delta$}
\put(-380, 69){$l_{x_1}$}
\put(-375,87){$\Sigma$}
\put(-401,28){$y_1$}
\put(-380,6){$l_{y_1}$}
\put(-271, 123){${z_1}$}
\put(-298, 151) {$l_{z_1}$}
\put(-120,118){$x_1$}
\put(-165,31){$3x_1$}
\put(-126,30){$l_{x_1}$}
\put(-71,12){$y_1$}
\put(-49,49){$3\gamma$}
\put(-32,71){$3\delta$}
\put(-15,45){$\Sigma$}
\put(-36,132){$z_1$}
\caption{A rational grope bounded by $3x_1$.}
\label{fig:Bor}
\end{figure}
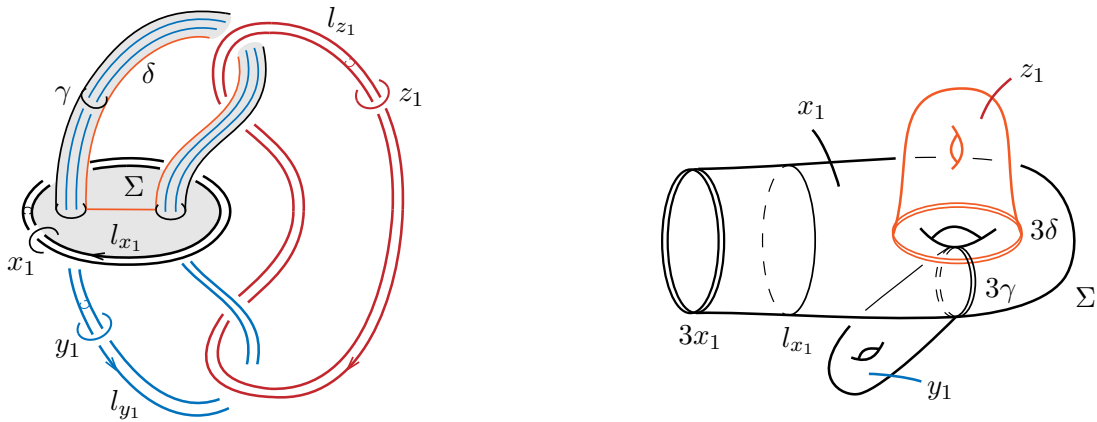

Let the filling solid tori be called $T_{x_1}, T_{y_1}, T_{z_1}$.
The curves $3x_1$ and $l_{x_1}$ cobound an annulus in $T_{x_1}$, and $l_{x_1}$ bounds a genus one surface $\Sigma$ in the complement of the other two solid tori, shown in the figure. A symplectic basis of curves on $\Sigma$ is given by curves $\gamma, \delta$ which are isotopic to $z_1, y_1$ respectively in the complement of the solid tori.

Analogously, $3x_2$ and the longitude $l_{x_2}$ of the corresponding unlinked component cobound an annulus in the filling solid torus $T_{x_2}$. Moreover, $l_{x_2}$ bounds a disk disjoint from everything else except for a single intersection point with $x_2$. 
The surface $\Sigma$ also intersects $x_1$ in a single point but with the opposite sign. 

\subsection{Lagrangians in $M_0$}
Consider the homology classes in $H_1(M_{\rm Bor})$ of the 1-cycles
\begin{equation} \label{eq: Lagrangian}
x:=x_1+x_2, \, y:=y_1+y_2, \, z:=z_1+z_2.
\end{equation}
By the discussion above, these three classes form a Lagrangian with respect to the torsion linking pairing. We will next show that $\lf([x],[y],[z])\neq 0\in \Q/Z$.

The curves $3\delta$ and $l_{z_1}$ cobound an annulus which goes over the solid torus $T_{z_1}$, and $l_{z_1}$ bounds a genus one surface in the complement of the other components of the Borromean rings. This is a second stage surface, intersecting $z$ in a single point, for the rational height 2 grope bounded by $3x$. The other second stage surface, intersecting $y$ in a single point, is bounded by $3\gamma$. (The second stage surfaces intersect $\Sigma$ and also each other; these intersections do not contribute to the calculation of $\lf$ and are not indicated on the right in Figure \ref{fig:Bor}.)

The formula \eqref{eq: triple linking definition} for computing $\lf([x],[y],[z])$  has a factor $\frac{1}{3}$ and there is a single non-trivial term contributing $\pm 1$, showing $\lf([x],[y],[z])$ is non-trivial. The sign is not important here, nevertheless checking the chosen orientations of the link components and their meridians shows that the sign is +1.

The above calculation shows that the triple linking form is non-trivial on the Lagrangian spanned by the elements \eqref{eq: Lagrangian}.
More generally, given any Lagrangian $L\subset H_1(M_0)$, consider a basis $l_1, l_2, l_3$ of $L$ which we will view as rows of a $(3\times 6)$-matrix $A_L$\footnote{Strictly speaking, the matrix depends on a choice of a basis of $L$. However, we omit the basis from the notation $A_L$ because the purpose of these matrices is to test the non-triviality of $\lambda_3$, a property that is basis-independent.}. The first three columns of this matrix correspond to the basis elements $x_1, y_1, z_1$ of $H_1(M_0)$, and the last three to $x_2, y_2, z_2$. For example, the basis of the Lagrangian in \eqref{eq: Lagrangian} corresponds to the $(3\times 6)$-matrix with two identity $3\times 3$ blocks.  Recall that based on the framing of the link in Figure \ref{fig:Example0}, the torsion linking and self-linking $\lambda(l_i, l_j)$ of the rows of the matrix $A_L$ is computed using the $\Q/\Z$-valued bilinear form with diagonal entries $\frac{1}{3}, \frac{1}{3}, \frac{1}{3}, -\frac{1}{3}, -\frac{1}{3}, -\frac{1}{3}$. 

It follows from the anti-symmetry of $\lf$ (Proposition \ref{prop: symmetry}) and from the calculation above (corresponding to the fact that the components with meridians $x_1, y_1, z_1$ form the Borromean rings) that $\lf(l_1, l_2, l_3)$ is a 3-form equal to the determinant of the left $3\times 3$ block of the matrix $A_L$. 

The following two matrices give an example of a dual pair of Lagrangians $L, L'$ with $\lf$ vanishing on both. 

\begin{equation} \label{eq: dual vanishing}
A_L:\; \begin{bmatrix}
  {0} &  {0} &  {0} &  {1} &  {1}  & {1}\\
 {0} &  {1} &  {-1} &  {0} &  {1}  & {-1}\\ 
 {1} &  {1} &  {1} &  {0} &  {0} & {0} \\ 
\end{bmatrix} \hspace{1.5cm} A_{L'}: \; \begin{bmatrix}
  {0} &  {0} &  {0} &  {-1} &  {1}  & {1}\\
 {0} &  {1} &  {-1} &  {0} &  {-1}  & {1}\\ 
 {-1} &  {1} &  {1} &  {0} &  {0} & {0} \\ 
\end{bmatrix}
\end{equation}

It follows that $\lf$ does not obstruct embeddability of $M_0$ into $S^4$ with this pair of dual Lagrangians. Next we construct an example of a rational homology sphere such that $\lf$ is non-trivial on one of the Lagrangians in {\em any} dual pair.

\subsection{A non-embeddable example} \label{sec: nonembeddable}

There are 48 Lagrangian subspaces $L\subset H_1(M_0)$ with a non-singular left $3\times 3$ block in the matrix $A_L$, and thus a non-vanishing $\lf$. However, there are other Lagrangians (to be precise, a total of 32) for which the determininant of this block is trivial in $\Z/3\Z$, see for example the pair $L, L'$ in \eqref{eq: dual vanishing}. A computer calculation shows that the Lagrangian subspaces of $H_1(M_0)$ form 1080 dual pairs.

Consider the following collection of links parametrized by vectors $v\in (\Z/3\Z)^{20}$. 
As in the beginning of Section \ref{sec: examples}, we start with the 6 component unlink, where half of the components have framing 3 and the other half $-3$. Rather than forming the Borromean rings just with the 3-framed components, now consider all possible triples of curves.
There are $\binom{6}{3}=20$ such choices, which we will also view as  triples of columns of the  $3\times 6$ matrices $A_L$. Consider the lexicographic order on these triples, so the first index corresponds to the leftmost three columns, the second index corresponds to the columns $1, 2, 4$, etc. 

Given $v\in (\Z/3\Z)^{20}$ and starting with an oriented 6 component unlink,
for each $i=1,\ldots, 20$ consider the $i$-th component $v_i$ of $v$. If $v_i=1$, perform a clasper surgery on the corresponding triple of components, or in other words take a band sum of the three components corresponding to the index $i$ with a separate, split copy of the Borromean rings oriented as in Figure \ref{fig:Bor}. If $v_i=-1$ then perform the same band sum, except that one of the bands is half-twisted (orientation-reversing). If $v_i=0$ then the corresponding three components are left intact.

Let $M_v$ denote the Dehn surgery on the resulting link. 
There are $3^{20}\approx 3.5$ billion links (and corresponding 3-manifolds) in this collection.  Figure \ref{fig:Example0} showed one possible link, corresponding to the vector with first component equal to 1 and the remaining 19 components equal to zero.

Note that for any $v$, the homology $H_1(M_v)$ and the torsion linking pairing are exactly the same as for the connected sum of three copies of $L(3,1)\# -L(3,1)$: adding the Borromean rings to the unlink does not affect $\lambda$. Accordingly, the Lagrangians in $H_1(M_v)=(\Z/3\Z)^6$ are also independent of $v$.
The calculation in Section \ref{sec: preliminary} implies that the triple linking form on $M_v$, evaluated on a basis of any Lagrangian $L$, equals the linear combination of the determinants of all $3\times 3$ sub-matrices of the matrix $A_L$, with coefficients given by the components of the vector $v$.

A computer calculation\footnote{A SageMath program is available at https://github.com/vkrushkal/Dual-pairs-of-Lagrangians/}
shows that the coefficient vector \[
v:=[-1, -1, 1, 1, 0, 0, 0, 0, 0, 0, -1, -1, -1, 1, 1, 0, 0, 0, 0, 0]
\]
has the property that the linear combination of $3\times 3$ determinants corresponding to the components of $v$ is non-trivial on at least one Lagrangian in each dual pair. It follows that the same conclusion holds for the triple linking form $\lambda_3$ on the 3-manifold $M_v$. We denote the rational homology sphere $M_v$ for this choice of $v$ by $M_{\rm Bor}$.

An application of Theorem \ref{thm: vanishing} concludes the proof of the following result, relying on the non-vanishing of the triple linking pairing.

\begin{lemma} \label{lem: example}
The 3-manifold $M_{\rm Bor}$ does not embed in an integer homology 4-sphere.  
\end{lemma}

{\em Remark.} To emphasize the subtlety of the above analysis we note that a computer calculation shows that for any vector $v\in(\Z/3\Z)^{20}$, the corresponding linear combination of determinants vanishes on some Lagrangian. In other words, for any $v$, the triple linking form on the 3-manifold $M_v$ is trivial on some Lagrangian in $H_1(M_v)$. A weaker condition of non-vanishing on at least one Lagrangian in each dual pair sufficed for our application to non-embeddability in $S^4$.

\section{Questions} \label{sec: questions}
We conclude by stating some questions motivated by our work.
\subsection{} \label{lf Heegaard}
It is shown in \cite[Theorem 4]{Freedman} that the condition on the linking form in Hantzsche's theorem is equivalent to the existence of an embedding of a Heegaard surface for $M^3$ in $S^3$ so that each of the two geometric Lagrangians defining the Heegaard splitting has trivial linking and self-linking numbers.  
It is a natural question whether there is an analogous characterization of the vanishing of all triple torsion linking invariants in terms of Milnor's invariants $\bar\mu_{1,2,3}$ for Heegaard diagrams.

\subsection{}
We expect that there is an obstruction theory of torsion linking forms of order $n$ for $n>3$ as well. The rational homology $3$-spheres in Section \ref{sec: examples} admit an immediate generalization to Bing doubles of the Borromean rings, providing examples whose embeddability in $S^4$ may be obstructed by such higher order invariants.

\subsection{} In the introduction we mentioned several higher order invariants, and in Section \ref{sec: relation Matsumoto} we showed that $\lf$ is related to one of them, the Matsumoto triple product. 
Note that for rational homology 3-spheres, $H^1(M;\Z)=0$ and so all Massey products vanish for integral cohomology. We do not know if our invariant is related to Massey products for cohomology with coefficients in $\Z/p\Z$.

\end{document}